\documentclass{svjour3}
\def\makeheadbox{}

\smartqed  % flush right qed marks, e.g. at end of proof
\usepackage[all]{xy}
\SelectTips{cm}{}

\usepackage{amssymb,mathptmx,amsmath}
\DeclareFontFamily{OML}{cyi}{} \DeclareFontShape{OML}{cyi}{m}{n}{
  <5> <6> <7> <8> <9> gen * wncyi
  <10> <10.95> <12> <14.4> <17.28> <20.74> <24.88> wncyi10
 }{}
\DeclareSymbolFont{rusletters}{OML}{cyi}{m}{n}
\DeclareSymbolFontAlphabet{\rusmath}{rusletters}
\DeclareMathSymbol\re{\rusmath}{rusletters}{"03}

\def\cprime{\/{\mathsurround=0pt$'$}}
\DeclareMathOperator{\sym}{sym}

\DeclareMathOperator{\Hom}{Hom}
\DeclareMathOperator{\id}{id}
\newcommand*{\eval}[1]{\left.#1\right|}

\newcommand{\ldb}{[\![}
\newcommand{\rdb}{]\!]}

\let\kappa=\varkappa
\let\phi=\varphi
\newcommand*{\sd}[2]{\{\,#1\mid#2\,\}}

\newcommand*{\SB}[2]{\ldb{#1},{#2}\rdb^{\mathrm{s}}}
\newcommand*{\NB}[2]{\ldb{#1},{#2}\rdb^{\mathrm{n}}}
\newcommand*{\RB}[2]{\ldb{#1},{#2}\rdb^{\mathrm{r}}}
\renewcommand{\kappa}{\varkappa}
\renewcommand{\phi}{\varphi}
\newcommand{\In}{\mathrm{i}}
\newcommand{\Ld}{\mathrm{L}}

\DeclareMathOperator{\Diff}{Diff}
\DeclareMathOperator{\Char}{char}
\DeclareMathOperator{\gr}{gr}
\DeclareMathOperator{\Can}{Can}
\DeclareMathOperator{\Ham}{Ham}
\allowdisplaybreaks[4]
\journalname{Acta Applicandae Mathematicae}
\begin{document}

\title{ALGEBRAIC THEORIES OF BRACKETS AND RELATED (CO)HOMOLOGIES\thanks{This
    work was supported in part by the NWO--RFBR grant 047.017.015,
    RFBR--Consortium E.I.N.S.T.E.I.N.\ grant 06-01-92060 and RFBR--CNRS grant
    08-07-92496.}} %\subtitle{}

\titlerunning{ALGEBRAIC THEORIES OF BRACKETS}

\author{I.~Krasil{\cprime}shchik}

\authorrunning{I.~Krasil{\cprime}shchik}

\institute{Iosif Krasil{\cprime}shchik \at Independent University of Moscow,
  B. Vlasevsky~11,
  119002 Moscow,  Russia\\
  \email{josephk@diffiety.ac.ru}}

\date{\ }

\maketitle

\begin{abstract}
  A general theory of the Fr\"olicher--Nijenhuis and Schouten--Nijenhuis
  brackets in the category of modules over a commutative algebra is
  described. Some related structures and (co)homology invariants are
  discussed, as well as applications to geometry.
\end{abstract}

\section{Introduction}
\label{sec:introduction}

Bracket structures play an important role in classical differential geometry
(see, for example, Refs.~\cite{F-N-1,F-N-2,S-N-3,R-N-1,R-N-2,S-N-1,S-N-2}),
Poisson geometry (e.g., \cite{Lich}), and the theory of integrable systems
(Refs.~\cite{K-K-V-lstar,KKVV,JK:SomeNew}). Being initially of a geometrical
nature, these brackets found exact counterparts in abstract algebra, in the
framework of Vinogradov's theory of algebraic differential
operators,~\cite{Vin:AlgLog} (see also book~\cite{KLV}). It became clear that
many geometrical constructions (such as the ones we meet in Hamiltonian
mechanics or in partial differential equations; cf.~with Refs.~\cite{WhatIsHF}
and~\cite{JK:FlConn}, respectively) may be more or less exactly expressed
using the language of commutative algebra.

In this paper, I collected together my old results on the algebraic theory of
the Fr\"olicher--Nijenhuis and Schouten--Nijenhuis brackets and related
homological and cohomological theories (for shortness, I call these brackets
Nijenhuis and Schouten ones). These results were initially published in
papers~\cite{JK:FlConn,JK:SchBr,JK:SomeNew,JK:HamCoh}. The results exposed
below are easily generalized to the case of super-commutative algebras
(see~\cite{JK:SupCan}) and, being slightly modified, can be incorporated in
Lychagin's ``colored calculus'' (see~Ref.\cite{Lych:Colored}).

To simplify exposition, I shall always assume that the algebra~$A$ is such
that the module~$\Lambda^1(A)$ (see Sec.~\ref{sec:algebr-calc-diff-1}) of
$1$-forms is projective and of finite type.

\section{A general scheme}
\label{sec:general-scheme}

This scheme was first presented in Ref.~\cite{Vin:Unification}.

Let~$\Bbbk$ be a field, $\Char\Bbbk\neq2$. Let
also~$\mathbb{P}=\sum_{k\in\mathbb{Z}}\mathbb{P}_k$,
$\mathbb{Q}=\sum_{k\in\mathbb{Z}}\mathbb{Q}_k$ be graded vector spaces
and~$\mathbb{Q}$ be endowed with a differential
\begin{equation*}
  d\colon\mathbb{Q}_k\to\mathbb{Q}_{k+1},\qquad d^2=0.
\end{equation*}
Assume that there exists a graded monomorphism
\begin{equation*}
  \phi\colon\mathbb{P}\to\Hom_\Bbbk^{\gr}(\mathbb{Q},\mathbb{Q}),\qquad
  \phi(\mathbb{P}_\alpha)\subset\Hom_\Bbbk^{\alpha+\beta}(\mathbb{Q},\mathbb{Q}),
  \quad  \beta=\gr\phi,
\end{equation*}
and define the ``\emph{Lie derivative}''
\begin{equation*}
  \Ld_p^\phi=[d,\phi(p)],\qquad p\in\mathbb{P}.
\end{equation*}
Here and everywhere below~$[\cdot\,,\cdot]$ denotes the graded commutator. If
we are lucky then for two elements~$p\in\mathbb{P}_\alpha$,
$p'\in\mathbb{P}_{\alpha'}$ we can define their \emph{$\phi$-bracket} by
\begin{equation*}
  \Ld_{\ldb p,p'\rdb}^\phi
  =[\Ld_p^\phi,\Ld_{p'}^\phi],\qquad
  \ldb p,p'\rdb\in\mathbb{P}_{\alpha+\alpha'+\beta+1}.
\end{equation*}
In some interesting cases we are lucky indeed.

\section{Algebraic calculus}
\label{sec:algebraic-calculus}

Let us introduce the basic notions of the calculus over commutative algebras
that will be needed below (see Refs.\cite{JK:ConciseIntr,KLV,Vin:AlgLog} for details).

\subsection{Differential operators}
\label{sec:algebr-calc-diff}

Consider a unitary commutative associative $\Bbbk$-algebra~$A$ and
$A$-modules~$P$ and~$Q$.
\begin{definition}\label{sec:diff-oper-1}
  A $\Bbbk$-linear map~$\Delta\colon P\to Q$ is a \emph{differential operator}
  (DO) of order~$\le k$ if
  \begin{equation*}
    [a_0,[a_1,[\dots[a_k,\Delta]\dots]]]=0
  \end{equation*}
  for all~$a_0,\dots,a_k\in A$.
\end{definition}
The set of all DOs~$P\to Q$ forms an $A$-bimodule denoted by
$\Diff_*(P,Q)$. An operator $X\colon A\to P$ is called a \emph{$P$-valued
  derivation} if
\begin{equation*}
  X(ab)=aX(b)+bX(a),\qquad a,b\in A.
\end{equation*}
The module of these derivations is denoted by~$D_1(P)$. Define by induction
the modules
\begin{equation*}
  D_i(P)=\sd{X\in D_1(D_{i-1}(P))}{X(a,b)+X(b,a)=0},\qquad i\ge2,
\end{equation*}
and set formally~$D_0(P)=P$. Elements of~$D_i(P)$ are called
\emph{multiderivations}.

\begin{remark}
  \label{sec:diff-oper}
  Let~$M$ be a smooth manifold and~$A=C^\infty(M)$. Let also~$\pi$ and~$\xi$
  be vector bundles over~$M$ and~$P=\Gamma(\pi)$, $Q=\Gamma(\xi)$ be the
  modules of their smooth sections. Then Definition~\ref{sec:diff-oper-1}
  gives the classical notion of a linear differential operator.
\end{remark}

\subsection{Differential forms}
\label{sec:algebr-calc-diff-1}

\begin{proposition}
  The correspondence~$P\Rightarrow D_i(P)$ is a representable functor from the
  category of $A$-modules to itself. The corresponding representative object
  is denoted by~$\Lambda^i=\Lambda^i(A)$ and called the module of
  \emph{differential $i$-forms} of the algebra~$A$. In particular\textup{,}
  there exists a natural derivation $d\colon A\to\Lambda^1$ such that any
  derivation $X\in D_1(P)$ uniquely decomposes as $X=\psi_X\circ d$\textup{,}
    %\begin{equation*}
    %        %  \xymatrix{
    %    A\ar[rr]^{X}\ar[rd]_{d}&&P\\
    %    &\Lambda^1\ar[ru]_{\psi_X}\rlap{,}&
    %        %  }
    %\end{equation*}
  where $\psi_X\in\Hom_A(\Lambda^1,P)$. The module~$\Lambda^1$ is generated by
  the elements of the form~$da$\textup{,} $a\in A$\textup{,} while~$\Lambda^i$
  are exterior powers of~$\Lambda^1$. This leads to the complex
  \begin{equation*}\xymatrixrowsep{0.2pc}
    \xymatrix{
      0\ar[r]&A\ar[r]^d&\Lambda^1\ar[r]&\dots\ar[r]&
      \Lambda^i\ar[r]^d&\Lambda^{i+1}\ar[r]&\dots
    }
  \end{equation*}
  called the \emph{de~Rham} complex of~$A$.
\end{proposition}

\subsection{Exterior products}
\label{sec:algebr-calc-exter}

Due to the above formulated proposition, the module
$\Lambda^*=\sum_i\Lambda^i$ is a Grassmannian algebra with the
\emph{exterior}, or \emph{wedge product}
\begin{equation*}
  \wedge\colon\Lambda^i\otimes_A\Lambda^j\to\Lambda^{i+j}.
\end{equation*}
A similar operation
\begin{equation*}
  \wedge\colon D_i(A)\otimes_A D_j(P)\to D_{i+j}(P)
\end{equation*}
is introduced by induction in $D_*(P)=\sum_i D_i(P)$. Namely, for~$i+j=0$ we
set
\begin{equation*}
  a\wedge p=ap,\qquad a\in D_0(A)=A,\quad p\in D_0(P)=P,
\end{equation*}
and
\begin{equation*}
  (X\wedge Y)(a)=X\wedge Y(a)+(-1)^jX(a)\wedge Y,\qquad i+j>0,
\end{equation*}
$X\in D_i(A)$, $Y\in D_j(P)$. In this way, $D_*(A)$ becomes a Grassmannian
algebra, $D_*(P)$ being a module over $D_*(A)$.

\subsection{Inner product}
\label{sec:algebr-calc-inner}

The \emph{inner product}
\begin{equation*}
  \In\colon D_i(P)\otimes_A\Lambda^j\to
  \begin{cases}
    P\otimes_A\Lambda^{j-i},&%X\otimes\omega\mapsto \In_X\omega,\quad
    j\ge i,\\
    D_{i-j}(P),&%X\otimes\omega\mapsto \In_\omega X,\quad
    j\le i,
  \end{cases}
\end{equation*}
is defined by induction. If~$i=0$ we set
\begin{equation*}
  \In(p\otimes\omega)=p\otimes\omega,\qquad p\in P=D_0(P),\quad
  \omega\in\Lambda^j,
\end{equation*}
and for~$j=0$ we set
\begin{equation*}
  \In(X\otimes a)=aX,\qquad a\in A=\Lambda^0,\quad X\in D_i(P).
\end{equation*}
If $i,j>0$ we set
\begin{equation*}
  \In(X\otimes da\wedge\omega)=\In(X(a)\otimes\omega).
\end{equation*}
We shall use the notation
\begin{equation*}
  \In_X\omega=
    \begin{cases}
      \In(X\otimes\omega),&i\ge j,\\
      0,&i<j,
    \end{cases}\quad\text{and}\quad
  \In_\omega X=
  \begin{cases}
    \In(X\otimes\omega),&i\le j,\\
    0,&i>j.
  \end{cases}
\end{equation*}

%\section{Algebraic calculus (inner product)}
%\label{sec:algebr-calc-inner-1}

\begin{remark}
  When $P=\Lambda^k$ and $j\ge i$, the inner product may be completed to the
  following operation
  \begin{equation*}
    \xymatrix{
      D_i(\Lambda^k)\otimes_A\Lambda^j\ar[r]^{\In}&
      \Lambda^k\otimes_A\Lambda^{j-i}\ar[r]^{\wedge}&\Lambda^{k+j-i}
    }
  \end{equation*}
  which will be also called the inner product.
\end{remark}

\begin{remark}
  Let $X\in D_i(P)$ and $\omega\in\Lambda^j$. Then the maps
  \begin{equation*}
    \In_X\colon\Lambda^*\to P\otimes_A\Lambda^*
  \end{equation*}
  and
  \begin{equation*}
    \In_\omega\colon D_*(P)\to D_*(P)
  \end{equation*}
  are super-differential operators of order $i$ and $j$, respectfully.
\end{remark}

\section{The Schouten bracket}
\label{sec:schouten-bracket}

We define here the Schouten bracket, describe its properties and related
(co)homologies. Some applications are also discussed.

\subsection{Definition and existence}
\label{sec:definition-existence}

Let $X\in D_i(A)$. Consider the Lie derivative
\begin{equation*}
  \Ld_X=d\circ\In_X-(-1)^i\In_X\circ d=
  [d,\In_X]\colon\Lambda^j\to\Lambda^{j-i}.
\end{equation*}

\begin{theorem}
  For any two elements $X\in D_i(A)$ and $X'\in D_{i'}(A)$ there exists a
  uniquely defined element $\SB{X}{X'}\in D_{i+i'-1}(A)$ such that
  \begin{equation*}
    [\Ld_X,\Ld_{X'}]=\Ld_{\SB{X}{X'}}.
  \end{equation*}
  This element is called the \emph{Schouten bracket} of $X$ and $X'$.
\end{theorem}

%\section{The Schouten bracket (existence)}
%\label{sec:scho-brack-exist}

\begin{proof}
  We establish existence of $\SB{X}{X'}$ by induction. For $i'=0$ we set
  \begin{equation*}
    \SB{X}{a}=X(a),\qquad a\in A=D_0(A),
  \end{equation*}
  and similarly for $i=0$
  \begin{equation*}
    \SB{a}{X'}=(-1)^{i'}X'(a).
  \end{equation*}
  If $i,i'>0$ we set
  \begin{equation*}
    \SB{X}{X'}(a)=\SB{X}{X'(a)}+(-1)^{i'-1}\SB{X(a)}{X'}.
  \end{equation*}
  It is easily checked that thus defined bracket enjoys the needed property.
\end{proof}

\subsection{Properties}
\label{sec:scho-brack-prop}

\begin{proposition}
  Let $X,X',X''\in D_*(A)$ be multiderivations of degree $i$\textup{,} $i'$
  and $i''$ respectively. Then\textup{:}
  \begin{enumerate}
  \item $\SB{X}{X'}+(-1)^{(i-1)(i'-1)}\SB{X'}{X}=0$\textup{,}
  \item $\SB{X}{\SB{X'}{X''}}=\SB{\SB{X}{X'}}{X''}
    +(-1)^{(i-1)(i'-1)}\SB{X'}{\SB{X}{X''}}$\textup{,}
  \item $\SB{X}{X'\wedge X''}= \SB{X}{X'}\wedge
    X''+(-1)^{(i-1)i'}X'\wedge\SB{X}{X''}$\textup{,}
  \item $\SB{X}{X'}=[X,X']$\textup{,} if $i=i'=1$\textup{,}
  \item $\In_{\SB{X}{X'}}=[\Ld_X,\In_{X'}]$.
  \end{enumerate}
\end{proposition}

\section{Poisson structures}
\label{sec:poisson-structures}

To any bivector~$\mathcal{P}\in D_2(A)$, one can put into correspondence a
skew-symmetric bracket $\{a,b\}_{\mathcal{P}}=\mathcal{P}(a,b)$, $a,b\in A$.
\begin{proposition}
  The following statements are equivalent\textup{:}
  \begin{enumerate}
  \item $\{a,b\}_{\mathcal{P}}$ satisfies the Jacobi identity\textup{;}
  \item $\SB{\mathcal{P}}{\mathcal{P}}=0$\textup{;}
  \item $\partial_{\mathcal{P}}\circ\partial_{\mathcal{P}}=0$\textup{,} where
    $\partial_{\mathcal{P}}=\SB{\mathcal{P}}{\cdot}$.
  \end{enumerate}
\end{proposition}

\begin{definition}
  If one of the previous conditions fulfills then:
  \begin{enumerate}
  \item $\mathcal{P}$ is called a \emph{Poisson structure} and a pair
    $(A,\mathcal{P})$ is a \emph{Poisson algebra}.
  \item $\{\cdot\,,\cdot\}_{\mathcal{P}}$ is the \emph{Poisson bracket}
    associated with $\mathcal{P}$.
  \item $\mathcal{P}(a)\in D_1(A)$ are \emph{Hamiltonian derivations}.
  \item Derivations $X$ satisfying $X\{a,b\}_{\mathcal{P}}=
    \{Xa,b\}_{\mathcal{P}}+\{a,Xb\}_{\mathcal{P}}$ are \emph{canonical
      derivations}.
  \end{enumerate}
\end{definition}

\subsection{Example: algebraic $T^*$ (see Ref.~\cite{WhatIsHF})}
\label{sec:example:-algebraic-t}

Let $\Diff_*(A)=\cup_{k\ge0}\Diff_k(A)$ denote the algebra of all DOs $A\to
A$. For any $\Delta\in\Diff_k(A)$ the coset
$[\Delta]_k=\Delta\bmod\Diff_{k-1}(A)$ is called its \emph{symbol}.

If $s_1=[\Delta_1]_{k_1}$, $s_2=[\Delta_1]_{k_2}$ are two symbols we define
their product by
\begin{equation*}
  s_1\cdot s_2=[\Delta_1\circ\Delta_2]_{k_1+k_2}
\end{equation*}
and their bracket by
\begin{equation*}
  \{s_1,s_2\}=[\Delta_1\circ\Delta_2-\Delta_2\circ\Delta_1]_{k_1+k_2-1}.
\end{equation*}
In such a way we obtain the \emph{algebra of symbols}
\begin{equation*}
  S_*(A)=\sum_k\frac{\Diff_k(A)}{\Diff_{k-1}(A)}.
\end{equation*}

\begin{proposition}
  The above introduced algebra of symbols $S_*(A)$ is a graded commutative
  algebra with a graded Poisson bracket $\{\cdot\,,\cdot\}$. In the case
  $A=C^\infty(M)$ it coincides with the algebra of smooth functions on $T^*M$
  polynomial along the fibers\textup{,} while the bracket is the one defined
  by the canonical symplectic form $\Omega=dp\wedge dq$.
\end{proposition}

\begin{remark}
  \label{sec:example:-algebraic-t-1}
  The parallel between geometrical constructions and the corresponding
  algebraic modules is even deeper, though perhaps not so straightforward. As
  an example, let us describe how the canonical form~$\rho=p\,dq$ is defined
  within the model under consideration (other illustrations can be found,
  e.g., in Refs.~\cite{JK:Char,JK:deltaLemma}).

  Note first that exactly in the same way as it was done above one can define
  symbols of arbitrary operators $\Delta\in\Diff_*(P,Q)$. Moreover, under the
  assumption of Sec.~\ref{sec:introduction} one has the isomorphism
  \begin{equation}\label{eq:2}
    S_*(P,Q)=S_*(A)\otimes_A\Hom_A(P,Q).
  \end{equation}
  Now, to define a $1$-form, one needs to evaluate all derivations on this
  form. Let~$X\colon S_*(A)\to R$ be such a derivation,~$R$ being an
  $S_*(A)$-module. Since $A=S_0(A)\subset S_*(A)$, one can consider the
  restriction $\bar{X}=\eval{\bar{X}}_A\colon A\to R$. Due to
  Eq.~\eqref{eq:2}, one has
  \begin{equation*}
    [\bar{X}]\in S_*(A,R)=S_*(A)\otimes_A R,
  \end{equation*}
  and we set
  \begin{equation*}
    \In_{\bar{X}}\rho=\mu_S(\bar{X}),
  \end{equation*}
  where $\mu_S\colon S_*(A)\otimes_A R\to R$ is the
  multiplication. Consequently, we can define the form~$\Omega=d\rho$, but for
  general algebras it may be degenerate, contrary to the geometric case.
\end{remark}

\begin{remark}
  \label{sec:example:-algebraic-t-2}
  It may be also appropriate to discuss another parallel here. Namely, in
  geometry, $1$-forms are identified with sections of the cotangent bundle. In
  algebra, the notion of section transforms to a homomorphism $\phi\colon
  S_*(A)\to A=S_0(A)$ such that $\eval{\phi}_{A}=\id$. Consider a $1$-form
  $\omega\in\Lambda^1(A)$ and let us define the corresponding ``section''
  $\phi_\omega\colon S_*(A)\to A$. To this end, recall that in the case
  when~$\Lambda^1(A)$ is a finite type projective module the algebra
  $\Diff_*(A)$ is generated by $A=\Diff_0(A)$ and~$D_1(A)$. Therefore, any
  element~$s\in S_k(A)$ is of the form
  $s=\sum[X_{\alpha_1}]\cdot\ldots\cdot[X_{\alpha_k}]$, where~$X_\alpha$ are
  derivations. Then we set
  \begin{equation*}
    \phi_\omega(s)=
    \sum\In_{X_{\alpha_1}}(\omega)\cdot\ldots\cdot\In_{X_{\alpha_k}}(\omega).
  \end{equation*}
  Conversely, let~$\phi\colon S_*(A)\to A$ be a homomorphism. To define the
  corresponding $1$-form~$\omega_\phi$, we need to evaluate an arbitrary
  derivation $X\colon A\to P$ at it, where $P$ is an $A$-module. But
  \begin{equation*}
    [X]\in S_*(A,P)=S_*(A)\otimes_A P,
  \end{equation*}
  and we set
  \begin{equation*}
    \In_X(\omega_\phi)=\mu_A(\phi\otimes\id_P([X])),
  \end{equation*}
  where $\mu_A\colon A\otimes_\Bbbk P\to P$ is the multiplication.
\end{remark}

\subsection{Poisson cohomologies}
\label{sec:poisson-cohomologies}

Let $(A,\mathcal{P})$ be a Poisson algebra. The sequence
\begin{equation*}
  \xymatrix{
    0\ar[r]&A\ar[r]^{\partial_{\mathcal{P}}}&D_1(A)\ar[r]&\dots\ar[r]&
    D_i(P)\ar[r]^{\partial_{\mathcal{P}}}&D_{i+1}(P)\ar[r]&\dots,
  }
\end{equation*}
where $\partial_{\mathcal{P}}=\SB{\mathcal{P}}{\cdot}$, is the \emph{Poisson
  complex} of $A$ and its cohomologies $H^i(A,\mathcal{P})$ are the
\emph{Poisson cohomologies}.

\begin{proposition}
  \begin{enumerate}
  \item $H^0(A;\mathcal{P})$ consists of Casimirs of $\mathcal{P}$ and
    coincides with the Poisson center of $A$.
  \item $H^1(A;\mathcal{P})=\Can(\mathcal{P})/\Ham(\mathcal{P})$.
    where $\Can(\mathcal{P})$ is the space of canonical derivations and
    $\Ham(\mathcal{P})$ consists of the Hamiltonian ones.
  \item $H^2(A;\mathcal{P})$ coincides with the set of classes of nontrivial
    infinitesimal deformations of the Poisson structure
    $\mathcal{P}$.
  \item $H^3(A;\mathcal{P})$ contains obstructions to prolongation of
    infinitesimal deformations up to formal ones.
  \end{enumerate}
\end{proposition}

\subsection{Poisson homologies (see Ref.~\cite{Bryl})}
\label{sec:poisson-homologies}

Take a Poisson algebra $(A,\mathcal{P})$ and consider the operator
\begin{equation*}
  d_{\mathcal{P}}=\Ld_{\mathcal{P}}=[d,\In_\mathcal{P}]\colon
  \Lambda^j\to\Lambda^{j-1}.
\end{equation*}
By definition of the Poisson structure, one has
\begin{equation*}
  2d_{\mathcal{P}}\circ d_{\mathcal{P}}=[d_{\mathcal{P}},d_{\mathcal{P}}]=
  [\Ld_{\mathcal{P}},\Ld_{\mathcal{P}}]=
  \Ld_{\SB{\mathcal{P}}{\mathcal{P}}}=0
\end{equation*}
and one gets the complex
\begin{equation*}
  \xymatrix{
    \dots\ar[r]&\Lambda^j\ar[r]^{d_{\mathcal{P}}}&\Lambda^{j-1}\ar[r]&\dots\ar[r]&
    \Lambda^1\ar[r]^{d_{\mathcal{P}}}&A\ar[r]&0,
  }
\end{equation*}
whose homologies $H_j(A,\mathcal{P})$ are called the \emph{Poisson homologies}
of $(A,\mathcal{P})$. The action of $d_{\mathcal{P}}$ is fully defined by the
following two properties:
\begin{equation*}
  d_{\mathcal{P}}(\omega\wedge\omega')=(d_{\mathcal{P}}\omega)\wedge\omega'+
  (-1)^j\omega\wedge d_{\mathcal{P}}\omega'
\end{equation*}
and
\begin{equation*}
  d_{\mathcal{P}}(a\,db)=\{a,b\}_{\mathcal{P}},\qquad a,b\in A.
\end{equation*}

\subsection{Hamiltonian filtrations}
\label{sec:hamilt-filtr}

Let $\mathcal{H}^1\subset D_*(A)$ be the ideal generated by Hamiltonian
derivations. Let
\begin{equation*}
  \mathcal{H}^p=
  \underbrace{\mathcal{H}^1\wedge\dots\wedge\mathcal{H}^1}_{p\ \text{times}}
\end{equation*}
be its powers. Then
\begin{equation*}
  D_*(A)=\mathcal{H}^0\supset\mathcal{H}^1\supset\dots\supset\mathcal{H}^p
  \supset\mathcal{H}^{p+1}\supset\dots
\end{equation*}
is a filtration that generates a spectral sequence for Poisson cohomologies.
In a dual way, the filtration
\begin{equation*}
  0\subset\mathcal{H}_1\subset\dots\subset\mathcal{H}_p\subset
  \mathcal{H}_{p+1}\subset\dots
\end{equation*}
where
\begin{equation*}
  \mathcal{H}_p=\sd{\omega\in\Lambda^*}{\In_{X_1}(\dots(\In_{X_p}(\omega))\dots)=0
    \ \ \forall X_i\in\Ham(\mathcal{P})}
\end{equation*}
gives rise to a spectral sequence for Poisson homologies.

\section{Extended Poisson bracket}
\label{sec:extend-poiss-brack}

The Poisson bracket defined by a Poissonian bi-vector~$\mathcal{P}$ can be
extended to a super-bracket on the Grassmannian algebra~$\Lambda^*(A)$.

%\subsection{Definition}
%\label{sec:extend-poiss-brack}

Consider the differential $\partial_{\mathcal{P}}\colon D_i(A)\to D_{i+1}(A)$
and a form $\omega\in\Lambda^j$. Then we have the map
\begin{equation*}
  \Ld_\omega^{\mathcal{P}}=[\partial_{\mathcal{P}},\In_\omega]
  \colon D_i(A)\to D_{i-j+1}(A).
\end{equation*}
\begin{proposition}
  For any two forms $\omega\in\Lambda^j$\textup{,} $\omega'\in\Lambda^{j'}$
  the equality
  \begin{equation*}
    \In_{\{\omega,\omega'\}_{\mathcal{P}}}=[\Ld_\omega^{\mathcal{P}},\In_{\omega'}]
  \end{equation*}
  uniquely determines a form
  $\{\omega,\omega'\}_{\mathcal{P}}\in\Lambda^{j+j'-1}$\textup{,} which is
  called their \emph{Poisson bracket}.
\end{proposition}

%\subsection{Properties}
%\label{sec:extend-poiss-brack-1}

\begin{proposition}
  The Poisson bracket of forms enjoys the following properties\textup:
  \begin{enumerate}
  \item $\{a,db\}_{\mathcal{P}}=-\{a,b\}_{\mathcal{P}}$\textup{,}
  \item $\{da,db\}_{\mathcal{P}}=d\{a,b\}_{\mathcal{P}}$\textup{,}
  \item $\{\omega,{\omega'}\wedge{\omega''}\}_{\mathcal{P}}
    =\{\omega,{\omega'}\}_{\mathcal{P}}\wedge{\omega''}+
    (-1)^{(j-1)j'}{\omega'}\wedge\{\omega,{\omega''}\}_{\mathcal{P}}$\textup{,}
  \item $\{\omega,\{{\omega'},{\omega''}\}_{\mathcal{P}}\}_{\mathcal{P}} =
    \{\{\omega,{\omega'}\}_{\mathcal{P}},{\omega''}\}_{\mathcal{P}}
    +(-1)^{(j-1)(j'-1)}\{{\omega'},\{\omega,{\omega''}\}_{\mathcal{P}}\}_{\mathcal{P}}$\textup{,}
  \item $\{\omega,{\omega'}\}_{\mathcal{P}}=
    -(-1)^{(j-1)(j'-1)}\{{\omega'},\omega\}_{\mathcal{P}}$\textup{,}
  \item $\Ld_{\{\omega,{\omega'}\}_{\mathcal{P}}}=[\Ld_\omega,\Ld_{\omega'}]$.
  \end{enumerate}
\end{proposition}

Note that the first three properties may be taken for the constructive
definition of the extended bracket.

\section{Commuting structures}
\label{sec:commuting-structures}

Two Poisson structures $\mathcal{P}$ and $\mathcal{P}'$ \emph{commute}, or are
\emph{compatible} if $\SB{\mathcal{P}}{\mathcal{P}'}=0$. This is equivalent to
\begin{equation*}
  \partial_{\mathcal{P}}\circ\partial_{\mathcal{P}'}=0
\end{equation*}
or to the fact that $\mu\mathcal{P}+\mu'\mathcal{P}'$ (the \emph{Poisson
  pencil}) is a Poisson structure for all $\mu$, $\mu'\in\Bbbk$.
  
The Magri scheme (see Ref~\cite{Magri}) that establishes existence of infinite
series of commuting conservation laws for \emph{bi-Hamiltonian systems} has an
exact algebraic counterpart:

\begin{theorem}
  Let $A$ be an algebra with two commuting Poisson structures ${\mathcal{P}}$
  and ${\mathcal{P}'}$ and assume that $H^1(A;{\mathcal{P}'})=0$. Assume also
  that two elements $a_1$\textup{,} $a_2\in A$ are given\textup{,} such that
  $\partial_{\mathcal{P}}(a_1)=\partial_{{\mathcal{P}'}}(a_2)$. Then\textup{:}
  \begin{enumerate}
  \item There exist elements $a_3,\dots,a_s,\ldots\in A$ satisfying
    \begin{equation*}
      \partial_{\mathcal{P}}(a_s)=\partial_{{\mathcal{P}'}}(a_{s+1}).
    \end{equation*}
  \item All elements $a_1,\dots,a_s,\dots$ are in involution with respect to
    both Poisson structures\textup{,} i.e.\textup{,}
    \begin{equation*}
      \{a_\alpha,a_\beta\}_{\mathcal{P}}=\{a_\alpha,a_\beta\}_{{\mathcal{P}'}}=0
    \end{equation*}
    for all $\alpha$\textup{,} $\beta\ge1$.
  \end{enumerate}
\end{theorem}

\section{The Nijenhuis bracket}
\label{sec:nijenh-brack-the}

Consider a form-valued derivation $\Omega\in D_1(\Lambda^k)$ and the Lie
derivative
\begin{equation*}
  \Ld_\Omega=[d,\In_\Omega]\colon\Lambda^j\to\Lambda^{k+j},
\end{equation*}
where $\In_\Omega$ is defined by the composition
\begin{equation*}
  \xymatrix{
    D_1(\Lambda^k)\otimes_A\Lambda^j\ar[r]^{\In}&
    \Lambda^k\otimes_A\Lambda^{j-1}\ar[r]^{\wedge}&\Lambda^{k+j-1}.
  }
\end{equation*}

\begin{proposition}
  The above Lie derivative possesses the following properties\textup{:}
  \begin{enumerate}
  \item $\Ld_\Omega(\omega\wedge\omega')=\Ld_\Omega(\omega)\wedge\omega'
    +(-1)^{kj}\omega\wedge\Ld_\Omega(\omega')$\textup{,}
  \item $[\Ld_\Omega,d]=0$\textup{,}
  \item
    $\Ld_{\omega\wedge\Omega}=\omega\wedge\Ld_\Omega+(-1)^{k+j}d\omega\wedge\In_\Omega$.
  \end{enumerate}
  Here $\omega\in\Lambda^j$\textup{,} $\omega'\in\Lambda^{j'}$\textup{,}
  $\Omega\in D_1(\Lambda^k)$.
\end{proposition}

Other basic properties of the Nijenhuis bracket are presented in the following

\begin{proposition}
  Let $\Omega\in D_1(\Lambda^j)$\textup{,} $\Omega'\in
  D_1(\Lambda^{j'})$\textup{,} $\Omega''\in D_1(\Lambda^{j''})$\textup{,} and
  $\omega\in\Lambda^i$. Then\textup{:}
  \begin{enumerate}
  \item $\NB{\Omega}{\Omega'}+(-1)^{jj'}\NB{\Omega'}{\Omega}=0$\textup{,}\\
  \item $\NB{\Omega}{\NB{\Omega'}{\Omega''}}=
    \NB{\NB{\Omega}{\Omega'}}{\Omega''}+
    (-1)^{jj'}\NB{\Omega'}{\NB{\Omega}{\Omega''}}$\textup{,}\\
  \item $\NB{\Omega}{\omega\wedge\Omega'}=\Ld_\Omega(\omega)\wedge\Omega'+
    (-1)^{ij}\omega\wedge\NB{\Omega}{\Omega'}-
    (-1)^{(j+1)(i+j')}d\omega\wedge\In_{\Omega'}(\Omega)$\textup{,}\\
  \item $[\Ld_{\Omega},\In_{\Omega'}]=(-1)^j\Ld_{\In_{\Omega'}}\Omega+
    \In_{\NB{\Omega}{\Omega'}}$\textup{,}\\
  \item $\In_\Omega\NB{\Omega'}{\Omega''}=\NB{\In_\Omega\Omega'}{\Omega''}+
    (-1)^{(j+1)j'}\NB{\Omega'}{\In_\Omega\Omega''}
    + (-1)^{j'}\In_{\NB{\Omega}{\Omega'}}\Omega''\\-
    (-1)^{(j'+1)j''}\In_{\NB{\Omega}{\Omega''}}\Omega'$.
  \end{enumerate}
\end{proposition}

On decomposable elements the Nijenhuis bracket acts as follows.  Let
$\omega\in\Lambda^j$, $\omega'\in\Lambda^{j'}$, $X$, $X'\in D_1(A)$. Then
\begin{align*}
  \NB{\omega\wedge X}{\omega'\wedge X'}%\\&
  &=\omega\wedge\omega'\wedge[X,X']+\Ld_{\omega\wedge X}(\omega')\wedge X'-
  (-1)^{jj'}\Ld_{\omega'\wedge X'}(\omega)\wedge X\\%[5mm]
  &=\omega\wedge\omega'\wedge[X,X']+\omega\wedge\Ld_X(\omega')\wedge X'-
  \Ld_{X'}(\omega)\wedge\omega'\wedge X\\
  &+(-1)^jd\omega\wedge\In_X(\omega')\wedge X' +(-1)^j\In_{X'}(\omega)\wedge
  d\omega'\wedge X.
\end{align*}

%\subsection{Nijenhuis complex}
%\label{sec:nijenhuis-complex}

A derivation $\mathcal{N}\in D_1(\Lambda^1)$ is called \emph{integrable} if
\begin{equation*}
  \NB{\mathcal{N}}{\mathcal{N}}=0.
\end{equation*}
With any integrable derivation one can associate a complex
\begin{equation*}
  \xymatrix{
    0\ar[r]&D_1(B)\ar[r]^{\partial_{\mathcal{N}}}&\dots\ar[r]&
    D_1(\Lambda^j(B))\ar[r]^{\partial_{\mathcal{N}}}&D_1(\Lambda^{j+1}(B))\ar[r]&\dots
  }
\end{equation*}
where $\partial_{\mathcal{N}}=\NB{\mathcal{N}}{\cdot}$.

Such structures, in particular, arise in an algebraic model of flat
connections.

\section{Flat connections}
\label{sec:flat-connections}

Let $A$ and $B$ be $\Bbbk$-algebras and $\gamma\colon A\to B$ be a
homomorphism. Then $B$ is an $A$-algebra and one can consider the module
$D_1(A,B)$ of $B$-valued derivations $A\to B$. For any $X\in D_1(B)$ denote by
$\eval{X}_A\in D_1(A,B)$ its restriction to $A$.

A \emph{connection} is a $B$-homomorphism $\nabla\colon D_1(A,B)\to D_1(B)$
such that $\eval{\nabla(X)}_{A}=X$. A vector-valued form $U_\nabla\in D_1(B)$
defined by
\begin{equation*}
  \In_X(U_\nabla)=X-\nabla(\eval{X}_A),\qquad X\in D_1(B),
\end{equation*}
is called the \emph{connection form}. For any two derivations $X$, $X'\in
D_1(A,B)$ we set
\begin{equation*}
  R_\nabla(X,X')=[\nabla(X),\nabla(X')]-\nabla(\nabla(X)\circ X'
  -\nabla(X')\circ X);
\end{equation*}
$R_\nabla$ is the \emph{curvature} of $\nabla$. A connection is \emph{flat} if
$R_\nabla=0$.

\subsection{Nijenhuis cohomologies associated to a connection}
\label{sec:nijenh-cohom-assoc}

\begin{theorem}
  Let $\nabla$ be a connection. Then
  \begin{equation*}
    \In_X(\In_{X'}(\NB{U_\nabla}{U_\nabla}))=2R_\nabla(\eval{X}_A,\eval{X'}_A)
  \end{equation*}
  for any $X$\textup{,} $X'\in D_1(B)$.
\end{theorem}
Hence, to any flat connection, i.e., to a connection whose curvature vanishes,
we associate a Nijenhuis complex with $\mathcal{N}=U_\nabla$. In applications,
its \emph{vertical} subcomplex is useful:
\begin{equation*}
  \xymatrix{
    0\ar[r]&D_1^v(B)\ar[r]^{\partial_{\mathcal{N}}}&\dots\ar[r]&
    D_1^v(\Lambda^j(B))\ar[r]^{\partial_{\mathcal{N}}}
    &D_1^v(\Lambda^{j+1}(B))\ar[r]&\dots,
  }
\end{equation*}
where $D_1^v(P)=\sd{X\in D_1(P)}{\eval{X}_A=0}$. Denote its cohomologies by
$H^j(B,\nabla)$.

\subsection{Nijenhuis cohomologies: $H^0$, $H^1$, and $H^2$}
\label{sec:nijenh-cohom-h0}

\begin{theorem}
  Let $\nabla$ be a flat connection. Then\textup{:}
  \begin{enumerate}
  \item The cohomology groups~$H^j(B,\nabla)$ inherit the inner product
    operation\textup{,}
    \begin{equation*}
      \In\colon H^j(B,\nabla)\times H^{j'}(B,\nabla)\to H^{j+j'-1}(B,\nabla).
    \end{equation*}
    In particular\textup{,} the group $H^1(B,\nabla)$ is an associative
    algebra represented in endomorphisms of~$H^0(B,\nabla)$\textup{:}
    \begin{gather*}
      \In\colon H^1(B,\nabla)\times H^1(B,\nabla)\to H^1(B,\nabla),\\
      \In\colon H^1(B,\nabla)\times H^0(B,\nabla)\to H^0(B,\nabla).
    \end{gather*}
  \item The cohomology groups~$H^j(B,\nabla)$ inherit the Nijenhuis
    bracket\textup{,}
    \begin{equation*}
      \NB{\cdot\,}{\cdot}\colon H^j(B,\nabla)\times H^{j'}(B,\nabla)\to
      H^{j+j'}(B,\nabla).
    \end{equation*}
    In particular\textup{,} $H^0(B,\nabla)$ is a Lie algebra\textup{:}
    \begin{equation*}
      \NB{\cdot\,}{\cdot}\colon H^0(B,\nabla)\times H^0(B,\nabla)\to
      H^0(B,\nabla).
    \end{equation*}
  \end{enumerate}
\end{theorem}

\subsection{Application to differential equations}
\label{sec:nijenh-cohom-appl}

Let $\mathcal{E}\subset J^\infty(\pi)$ be an infinitely prolonged differential
equation in the jet bundle of a bundle $\pi\colon E\to M$. The bundle
$\pi_\infty\colon\mathcal{E}\to M$ is always endowed with a natural flat
connection $\mathcal{C}$ (the \emph{Cartan connection}, see
Refs~\cite{B-Ch-D-AMS,KLV}) and taking
\begin{equation*}
  \gamma=\pi_\infty^*\colon A=C^\infty(M)\to B=C^\infty(\mathcal{E})
\end{equation*}
we obtain the picture considered above.

Let us use the notation $H^j(\mathcal{E},\mathcal{C})$ for the cohomology
groups arising in this case.

\subsection{The main result}
\label{sec:appl-diff-equat}

\begin{theorem}
  For any formally integrable equation $\mathcal{E}$ that surjectively
  projects to $J^0(\pi)$ one has\textup{:}
  \begin{enumerate}
  \item $H^0(\mathcal{E},\mathcal{C})$ coincides with the Lie algebra
    $\sym\mathcal{E}$ of higher symmetries of~$\mathcal{E}$.
  \item Elements of $H^1(\mathcal{E},\mathcal{C})$ act on $\sym\mathcal{E}$
    and thus are identified with recursion operators for symmetries.
  \item On the other hand\textup{,} elements of $H^1(\mathcal{E},\mathcal{C})$
    can be understood as classes of nontrivial infinitesimal deformations of
    the equation structure.
  \item $H^2(\mathcal{E},\mathcal{C})$ contains obstructions to prolongation
    of infinitesimal deformations up to formal ones.
  \end{enumerate}
\end{theorem}

\subsection{Commutative hierarchies}
\label{sec:comm-hier}

Let $(B,\nabla)$ be an algebra with flat connection. For $X=X_0\in
H^0(B,\nabla)$ and $R\in H^1(B,\nabla)$, use the notation $R(X)=\In_X(R)$,
$X_n=R^n(X)$, $n=0,1,2\dots$

\begin{theorem}
  Assume that $H^2(B,\nabla)=0$. Then for any $X$\textup{,} $Y\in
  H^0(B,\nabla)$ and $R\in H^1(B,\nabla)$ for all $m$\textup{,}
  $n\in\mathbb{Z}_+$ one has
  \begin{equation*}
    [X_m,Y_n]=[X,Y]_{m+n}+\sum_{i=0}^{n-1}(\NB{X}{R}(Y_i))_{m+n-i-1}-
    \sum_{j=0}^{m-1}(\NB{Y}{R}(X_j))_{m+n-j-1}.
  \end{equation*}
\end{theorem}

\begin{corollary}
  If $\NB{X}{R}=\NB{Y}{R}=0$ and $[X,Y]=0$ then $[X_m,Y_n]=0$ for all
  $m$\textup{,} $n\in\mathbb{Z}_+$.
\end{corollary}

\begin{remark}
  If $\mathcal{E}$ is a scalar evolutionary equation of order $>1$ then
  $H^2(\mathcal{E},\mathcal{C})=0$.
\end{remark}

\subsection{Bi-complex}
\label{sec:bi-complex}

Let $\mathcal{N}\in D_1(\Lambda^1)$ be an integrable element, i.e.,
$\NB{\mathcal{N}}{\mathcal{N}}=0$. Then the operator
\begin{equation*}
  d_{\mathcal{N}}=\Ld_{\mathcal{N}}\colon\Lambda^j\to\Lambda^{j+1}
\end{equation*}
is a differential: $d_{\mathcal{N}}\circ d_{\mathcal{N}}=0$. Moreover, one has
\begin{equation*}
  [d,d_{\mathcal{N}}]=0
\end{equation*}
and consequently the pair $(d_{\mathcal{N}},\bar{d}_{\mathcal{N}})$, where
$\bar{d}_{\mathcal{N}}=d-d_{\mathcal{N}}$, constitutes a bi-complex that
converges to the de~Rham cohomologies of~$B$.

In the case of differential equations ($A=C^\infty(M)$,
$B=C^\infty(\mathcal{E})$, and $\mathcal{N}$ is the connection form of the
Cartan connection in the bundle $\pi_\infty\colon \mathcal{E}^\infty\to M$),
this bi-complex coincides with the \emph{variational bi-complex}, or
\emph{Vinogradov's $\mathcal{C}$-spectral sequence},
see~Ref.~\cite{V,Vin:C-Spec}.

\section{More brackets\dots}
\label{sec:more-brackets}

To conclude, note that several more brackets can be constructed in a similar
way.
\begin{enumerate}
\item First, mention the \emph{Nijenhuis--Richardson bracket}
  \begin{equation*}
    \RB{\cdot\,}{\cdot}\colon D_1(\Lambda^j)\otimes_A
    D_1(\Lambda^{j'})\to D_1(\Lambda^{j+j'-1})
  \end{equation*}
  that can be defined by
  \begin{equation*}
    [\In_\Omega,\In_{\Omega'}]=\In_{\RB{\Omega}{\Omega'}}
  \end{equation*}
  and is of the form
  \begin{equation*}
    \RB{\Omega}{\Omega'}=\In_\Omega(\Omega')-
    (-1)^{(i-1)(j-1)}\In_{\Omega'}(\Omega)
  \end{equation*}
  and is one of the classical and well known brackets.
\end{enumerate}

Two more brackets arise also if we fix a Poisson structure $\mathcal{P}\in
D_2(A)$ or a Nijenhuis structure $\mathcal{N}\in D_1(\Lambda)$:
\begin{enumerate}
\item[2.] Consider the inner product
  \begin{equation*}
    \In\colon D_i(\Lambda^1)\otimes_A D_k(A)\to D_{i+k-1}(A).
  \end{equation*}
  Then for $\Omega\in D_i(\Lambda^1)$ the following ``Lie derivative'' arises:
  \begin{equation*}
    \Ld_\Omega^{\mathcal{P}}=[\partial_{\mathcal{P}},\In_\Omega]\colon
    D_k(A)\to D_{k+i}(A)
  \end{equation*}
  and one can introduce a bracket
  \begin{equation*}
    \ldb\cdot\,,\cdot\rdb^{\mathcal{P}}\colon D_i(\Lambda^1)\times
    D_{i'}(\Lambda^1)\to D_{i+i'}(\Lambda^1)
  \end{equation*}
  by
  \begin{equation*}
    \Ld_{\ldb\Omega,\Omega'\rdb^{\mathcal{P}}}^{\mathcal{P}}=
    [\Ld_\Omega^{\mathcal{P}},\Ld_{\Omega'}^{\mathcal{P}}].
  \end{equation*}
\end{enumerate}

%\section{More brackets}
%\label{sec:more-brackets-1}

\begin{enumerate}
\item[3.] In a similar way, one can consider the inner product
  \begin{equation*}
    \In\colon D_1(\Lambda^i)\otimes_A D_1(\Lambda^k)
    \to D_1(\Lambda^{i+k-1})
  \end{equation*}
  and the ``Lie derivative''
  \begin{equation*}
    \Ld_\Omega^{\mathcal{N}}=[\partial_{\mathcal{N}},\In_\Omega]\colon
    D_1(\Lambda^k)\to D_1(\Lambda^{k+i}),
  \end{equation*}
  $\Omega\in D_1(\Lambda^i)$. Then a new bracket
  \begin{equation*}
    \ldb\cdot\,,\cdot\rdb^{\mathcal{N}}\colon D_1(\Lambda^i)\times
    D_1(\Lambda^{i'})\to D_1(\Lambda^{i+i'})
  \end{equation*}
  is defined by
  \begin{equation*}
    \Ld_{\ldb\Omega,\Omega'\rdb^{\mathcal{N}}}^{\mathcal{N}}=
    [\Ld_\Omega^{\mathcal{N}},\Ld_{\Omega}^{\mathcal{N}}].
  \end{equation*}
\end{enumerate}

\section{\dots and when brackets fail to arise}
\label{sec:when-brackets-fail}

One can also define the inner products
\begin{equation*}
  \In\colon D_i(\Lambda^j)\otimes_A\Lambda^k\to\Lambda^{k+j-i}
\end{equation*}
and
\begin{equation*}
  \In\colon D_i(\Lambda^j)\otimes_A D_k(A)\to D_{k-j+i}(A)
\end{equation*}
together with the corresponding Lie actions
\begin{equation*}
  \Ld_\Omega=[d,\In_\Omega]\colon\Lambda^k\to\Lambda^{k+j-i+1}
\end{equation*}
and
\begin{equation*}
  \Ld_\Omega^{\mathcal{P}}=[\partial_{\mathcal{P}},\In_\Omega]\colon
  D_k(A)\to D_{k-j+i+1}(A),
\end{equation*}
where $\Omega\in D_i(\Lambda^j)$.  Of course, it is tempting to find the
elements $\ldb\Omega,\Omega'\rdb$ and $\ldb\Omega,\Omega'\rdb^{\mathcal{P}}$
such that
\begin{equation*}
  \Ld_{\ldb\Omega,\Omega'\rdb}=[\Ld_\Omega,\Ld_{\Omega'}],\qquad
  \Ld_{\ldb\Omega,\Omega'\rdb^{\mathcal{P}}}^{\mathcal{P}}=
  [\Ld_\Omega^{\mathcal{P}},\Ld_{\Omega'}^{\mathcal{P}}],
\end{equation*}
but in general such elements do not exist (see discussion of these matters in
Ref.~\cite{Vin:Unification}).

\end{document}